\documentclass[10pt]{amsart}
\usepackage{amssymb,MnSymbol}
\usepackage{amsthm,amsmath}
\usepackage{cite}

\title[$n$-ary semigroups polynomial-derived from integral domains]{A description of $n$-ary semigroups polynomial-derived from integral domains}\thanks{Communicated by Mikhail Volkov.}

\author{Jean-Luc Marichal}\thanks{Corresponding author: Jean-Luc Marichal (Phone: +352 466644-6662).}
\address{Mathematics Research Unit, FSTC, University of Luxembourg \\
6, rue Coudenhove-Kalergi, L-1359 Luxembourg, Luxembourg}
\email{jean-luc.marichal[at]uni.lu }

\author{Pierre Mathonet}
\address{Mathematics Research Unit, FSTC, University of Luxembourg \\
6, rue Coudenhove-Kalergi, L-1359 Luxembourg, Luxembourg}
\email{pierre.mathonet[at]uni.lu }

\date{January 27, 2011}

\begin{document}

\theoremstyle{plain}
\newtheorem{theorem}{Theorem}%[section]% Supprimer [section] pour une numérotation linéaire
\newtheorem{lemma}[theorem]{Lemma}
\newtheorem{proposition}[theorem]{Proposition}
\newtheorem{corollary}[theorem]{Corollary}
\newtheorem{fact}[theorem]{Fact}
\newtheorem*{main}{Main Theorem}

\theoremstyle{definition}
\newtheorem{definition}[theorem]{Definition}
\newtheorem{example}[theorem]{Example}

\theoremstyle{remark}
\newtheorem*{conjecture}{Conjecture}
\newtheorem{remark}{Remark}
\newtheorem{claim}{Claim}

\newcommand{\N}{\mathbb{N}}
\newcommand{\Z}{\mathbb{Z}}
\newcommand{\Q}{\mathbb{Q}}
\newcommand{\R}{\mathbb{R}}
\newcommand{\bfx}{\mathbf{x}}

\begin{abstract}
We provide a complete classification of the $n$-ary semigroup structures defined by polynomial functions over infinite commutative integral domains
with identity, thus generalizing G{\l}azek and Gleichgewicht's classification of the corresponding ternary semigroups.
\end{abstract}



\maketitle

%---------------------------------------------------------------------------------------------- Section 1
\section{Introduction}

Let $R$ be an infinite commutative integral domain with identity and let $n\geqslant 2$ be an integer. In this note we provide a complete
description of all the $n$-ary semigroup structures defined by polynomial functions over $R$ (i.e., the $n$-ary semigroup structures
polynomial-derived from $R$).
% As usual, we denote by $0$ and $1$ the zero and identity elements of $R$.

For any integer $k\geqslant 1$, let $[k]=\{1,\ldots,k\}$. Recall that a function $f\colon R^n\to R$ is said to be \emph{associative} if it
solves the following system of $n-1$ functional equations:
\begin{eqnarray}
\lefteqn{f(x_1\ldots, f(x_i,\ldots,x_{i+n-1}),\ldots,x_{2n-1})}\nonumber\\
&=& f(x_1\ldots, f(x_{i+1},\ldots,x_{i+n}),\ldots,x_{2n-1}),\qquad i\in [n-1].\label{eq:assoc}
\end{eqnarray}
In this case, the pair $(R,f)$ is called an \emph{$n$-ary semigroup}.

The introduction of $n$-ary semigroups goes back to D\"ornte \cite{Do1928} and led to the generalization of groups to $n$-ary groups (polyadic groups). The next extensive study on polyadic groups was due to Post \cite{Post2} and was followed by several contributions towards the description of $n$-ary groups and similar structures. To mention a few, see \cite{Dudek95,Dudek01,DGG77,G,GG67,H63,MonkSioson71,Zup67}.

We now state our main result, which provides a description of the possible associative polynomial functions from $R^n$ to
$R$. Let $\mathrm{Frac}(R)$ denote the fraction field of $R$ and let $\bfx=(x_1,\ldots,x_n)$.

\begin{main}\label{thm:111}
A polynomial function $p\colon R^n\to R$ is associative if and only if it is one of the following functions:
\begin{enumerate}
\item[$(i)$] $p(\bfx)=c$, where $c\in R$,

\item[$(ii)$] $p(\bfx)=x_1$,

\item[$(iii)$] $p(\bfx)=x_n$,

\item[$(iv)$] $p(\bfx)=c+\sum_{i=1}^n x_i$, where $c\in R$,

\item[$(v)$] $p(\bfx)=\sum_{i=1}^n \omega^{i-1}\, x_i$ (if $n\geqslant 3$), where $\omega\in R\setminus\{1\}$ satisfies $\omega^{n-1}=1$,

\item[$(vi)$] $p(\bfx)=-b+a\prod_{i=1}^n(x_i+b)$, where $a\in R\setminus\{0\}$ and $b\in\mathrm{Frac}(R)$ satisfy $ab^n-b\in R$ and $ab^k\in R$
for every $k\in [n-1]$.
\end{enumerate}
\end{main}

The Main Theorem shows that the associative polynomial functions of degree greater than $1$ are symmetric, i.e., invariant under any
permutation of the variables.

\begin{example}
The third-degree polynomial function $p\colon\Z^3\to\Z$ defined on the ring $\Z$ of integers by
\[
p(x_1,x_2,x_3)=9\, x_1x_2x_3+3\, (x_1x_2+x_2x_3+x_3x_1)+x_1+x_2+x_3
\]
is associative since it is the restriction to $\Z$ of the associative polynomial function $q\colon\Q^3\to\Q$ defined on the field $\Q$ of
rationals by
\[
q(x_1,x_2,x_3)=-\frac{1}{3}+9\prod_{i=1}^3\Big(x_i+\frac{1}{3}\Big).
\]
\end{example}

The classification given in the Main Theorem was already obtained for ternary semigroups (i.e., when $n=3$) by G{\l}azek and Gleichgewicht
\cite{GG85}. Surprisingly, the classification for arbitrary $n$ remains essentially the same except for certain solutions of type $(v)$ (already
mentioned in \cite{Do1928}), whose existence is subordinate to that of nontrivial roots of unity. Note that, when $n$ is odd, $(v)$ always provides
the solution
$$
p(\bfx)=\sum_{i=1}^n (-1)^{i-1}\, x_i\, ,
$$
which was the unique solution of type $(v)$ found in \cite{GG85} for $n=3$.

%The classification of ternary semigroups from \cite{GlaGle85} was obtained by first showing that an associative ternary polynomial function must
%be of degree at most one in each variable and then solving explicitly the associativity equations for a general function of this form.
%
%In Section 2 we proceed similarly and first show that, for general $n$, an associative $n$-ary polynomial function must be of degree at most
%one in each variable. From the associativity equations we then derive conditions on the coefficients of such functions. We select an appropriate
%subset of these conditions to show that associative polynomial functions are either of degree at most one or symmetric. In the former case we
%then obtain easily solutions $(i)$--$(v)$ of the Main Theorem and in the latter we obtain solution $(vi)$ by showing that the associative polynomial
%functions are restrictions to $R$ of a multiple of the product function, up to an affine transformation in $\mathrm{Frac}(R)$.

In Section 2 we give the proof of the Main Theorem. In Section 3 we analyze some properties of these $n$-ary semigroup structures: we show that they are medial, determine the $n$-ary groups
defined by polynomial functions, and discuss irreducibility issues for these $n$-ary semigroups.

%---------------------------------------------------------------------------------------------- Section 2
\section{Technicalities and proof of the Main Theorem}

Throughout this section, with every function $f\colon R^n\to R$ we associate $n$ functions $f_i\colon R^{2n-1}\to R$, $i\in [n]$, defined by
\begin{equation}\label{eq:qi}
f_i(x_1,\ldots,x_{2n-1})=f(x_1\ldots, f(x_i,\ldots,x_{i+n-1}),\ldots,x_{2n-1}).
\end{equation}
It follows that $f$ is associative if and only if $f_1=\cdots =f_n$.

It is clear that the definition of $R$ enables us to identify the ring $R[x_1,\ldots,x_n]$ of polynomials
of $n$ indeterminates over $R$ with the ring of polynomial functions of $n$ variables from $R^n$ to $R$. Recall that, for any integer
$d\geqslant 0$, a polynomial function $p\colon R^n\to R$ of degree $\leqslant d$ can be written as
$$
p(\bfx)=\sum_{j_1+\cdots + j_n\leqslant d}c_{j_1,\ldots,j_n}\, x_1^{j_1}\cdots\, x_n^{j_n},\qquad c_{j_1,\ldots,j_n}\in R.
$$
For every $i\in [n]$, we denote the degree of $p$ in $x_i$ by $\mathrm{deg}(p,x_i)$. We also denote the degree of $p$ by $\mathrm{deg}(p)$.

\begin{proposition}\label{prop:degrees}
For every associative polynomial function $p\colon R^n\to R$, we have $\mathrm{deg}(p,x_i)\leqslant 1$ for every $i\in [n]$. Moreover, if
$\mathrm{deg}(p,x_n)=0$ (resp.\ $\mathrm{deg}(p,x_1)=0$), then $p$ is either a constant or the projection on the first (resp.\ the last)
coordinate.
\end{proposition}

\begin{proof}
Let $p\colon R^n\to R$ be an associative polynomial function and let $p_1,\ldots,p_n$ be the polynomial functions associated with $p$ as defined
in (\ref{eq:qi}). For every $i\in [n]$, we let $d_i=\mathrm{deg}(p,x_i)$. By associativity, we have
$$
\begin{array}{ll}
  d_1=\mathrm{deg}(p_i,x_1)=\mathrm{deg}(p_1,x_1)=d_1^2\, , & i\in [n]\setminus\{1\},\\
  d_n=\mathrm{deg}(p_i,x_{2n-1})=\mathrm{deg}(p_n,x_{2n-1})=d_n^2\, , & i\in [n]\setminus\{n\},
\end{array}
$$
which shows that $d_1\leqslant 1$ and $d_n\leqslant 1$.

Again by associativity, we have
$$
\begin{array}{ll}
  d_id_{n-i+1}=\mathrm{deg}(p_i,x_n)=\mathrm{deg}(p_1,x_n)=d_1d_{n}\, , & i\in [n],\\
  d_i=\mathrm{deg}(p_{i+1},x_i)=\mathrm{deg}(p_i,x_i)=d_1d_{i}\, , & i\in [n-1],\\
  d_{i}=\mathrm{deg}(p_{i-1},x_{n+i-1})=\mathrm{deg}(p_i,x_{n+i-1})=d_id_{n}\, , & i\in [n]\setminus\{1\}.
\end{array}
$$
If $d_1=d_n=1$, then the first set of equations shows that $d_i=1$ for every $i\in [n]$. If $d_n=0$, then the third set of equations shows that
$p$ is of the form $p(\bfx)=c_1x_1+c_0$ and hence we can conclude immediately. We proceed similarly if $d_1=0$.
\end{proof}

By Proposition \ref{prop:degrees} an associative polynomial function $p\colon R^n\to R$ can always be written in the form
$$
p(\bfx)=\sum_{j_1,\ldots,j_n\in\{0,1\}}c_{j_1,\ldots,j_n}\, x_1^{j_1}\cdots\, x_n^{j_n},\qquad c_{j_1,\ldots,j_n}\in R.
$$
Using subsets of $[n]$ instead of Boolean indexes, we obtain
\begin{equation}\label{eq:sdfa987}
p(\bfx)=\sum_{J\subseteq [n]}c_J\,\prod_{j\in J}x_j,\qquad c_J\in R.
\end{equation}

In order to prove the Main Theorem, we only need to describe the class of associative polynomial functions of the form (\ref{eq:sdfa987}). %We note that,
%from now on, the results of this section hold true as soon as the vector space of polynomials of degree at most one in each indeterminate is in
%bijection with the vector space of such polynomial functions through the substitution isomorphism. This is the case when $R$ is any commutative integral ring with identity.

To avoid cumbersome notation, for every subset $S=\{j_1,\ldots,j_k\}$ of integers and every integer $m$, we set $S+m=\{j_1+m,\ldots,j_k+m\}$.
Also, for every $i\in [n]$, we let
\begin{eqnarray*}
A_i &=& \{1,\ldots,i-1\} ~=~ [i-1],\\
B_i &=& \{i,\ldots,i+n-1\} ~=~ [n]+i-1,\\
C_i &=& \{i+n,\ldots,2n-1\} ~=~ [n-i]+n+i-1,
\end{eqnarray*}
with the convention that $A_1=C_n=\varnothing$.

\begin{lemma}\label{lemma:8sdf7}
If $p\colon R^n\to R$ is of the form (\ref{eq:sdfa987}), then for every $i\in [n]$ the associated function $p_i\colon R^{2n-1}\to R$ is of the
form
\begin{equation}\label{eq:ssdf53g}
p_i(x_1,\ldots,x_{2n-1})=\sum_{S\subseteq [2n-1]}r^i_S\,\prod_{j\in S}x_j
\end{equation}
and its coefficients are given in terms of those of $p$ by
$$
r^i_S=
\begin{cases}
c_{J_S^i\cup\{i\}}\, c_{K_S^i}\, , & \mbox{if $S\cap B_i\neq\varnothing$},\\
c_{J_S^i\cup\{i\}}\, c_{\varnothing}+c_{J_S^i}\, , & \mbox{otherwise},
\end{cases}
$$
where $J_S^i=(S\cap A_i)\cup((S\cap C_i)-n+1)$ and $K_S^i=(S\cap B_i)-i+1$.
\end{lemma}

\begin{proof}
We first note that
\[p(x_i,\ldots,x_{n+i-1})=\sum_{K\subseteq [n]}c_K\prod_{k\in K}x_{k+i-1}=\sum_{K\subseteq [n]}c_K\prod_{k\in K+i-1}x_{k}.\]
Then, partitioning $J\subseteq [n]$ into $J\cap A_i$, $J\cap\{i\}$, and $J\cap (C_i-n+1)$, we obtain
\begin{eqnarray*}
p_i(x_1,\ldots,x_{2n-1})
&=& \sum_{J\subseteq [n]}c_J\,\prod_{j\in J\cap A_i}x_j\prod_{j\in (J+n-1)\cap C_i}x_j\prod_{j\in J\cap\{i\}}\bigg(\sum_{K\subseteq [n]}c_K\,\prod_{k\in K+i-1}x_k\bigg)\\
&=& \sum_{J\subseteq [n],J\ni i}~\sum_{K\subseteq [n]}c_J\, c_K\,\prod_{j\in J\cap A_i}x_j\prod_{j\in (J+n-1)\cap C_i}x_j\prod_{k\in
K+i-1}x_k\\
&& \null + \sum_{J\subseteq [n],J\not\ni i}c_J\, \prod_{j\in J\cap A_i}x_j\prod_{j\in (J+n-1)\cap C_i}x_j.
\end{eqnarray*}
The result is then obtained by reading the coefficient of $\prod_{i\in S}x_i$ in the latter expression.
\end{proof}

\begin{proposition}\label{degone}
Let $p\colon R^n\to R$ be an associative polynomial function of the form (\ref{eq:sdfa987}). If $c_{[n]}=0$, then $\mathrm{deg}(p)\leqslant 1$.
\end{proposition}

\begin{proof}
We assume that $c_{[n]}=0$ and prove by induction that $c_J=0$ for every $J\subseteq [n]$ such that $|J|\geqslant 2$. Suppose that $c_J=0$ for
every $J\subseteq [n]$ such that $|J|\geqslant k$ for some $k\geqslant 3$. Fix $J_0\subseteq [n]$ such that $|J_0|=k-1$. We only need to show
that $c_{J_0}=0$.

Assume first that $\ell=\min(J_0)\leqslant (n+1)/2$.
\begin{enumerate}
\item[$(i)$] Case $\ell =1$. Let
$$
S=J_0\cup\big(\big(J_0+n-1\big)\setminus \{n\}\big)\subseteq [2n-1].
$$
We have $S\cap A_1=\varnothing$, $S\cap B_1=J_0$, and $(S\cap C_1)-n+1=J_0\setminus\{1\}$. By Lemma~\ref{lemma:8sdf7}, we have
$r_S^1=c_{J_0}^2$. Setting $m=\min([n]\setminus J_0)$,\footnote{In fact, $m=\mathrm{mex}_{[n]}(J_0)$, where `$\mathrm{mex}$' stands for the
\emph{minimal excluded number}, well known in combinatorial game theory.} we also have $|S\cap A_m|=|A_m|=m-1$ and $|(S\cap
C_m)-n+1|=|J_0|-(m-1)$. Moreover, $S\cap B_m\neq\varnothing$ for otherwise we would have $J_0=\{1\}$, which contradicts $|J_0|\geqslant 2$.
Thus, using Lemma~\ref{lemma:8sdf7}, associativity, and the induction hypothesis, we have $r_S^m=0$ and therefore
$$
c_{J_0}^2=r_S^1=r_S^m=0.
$$

\item[$(ii)$] Case $1<\ell\leqslant (n+1)/2$. Let
$$
S=(J_0+\ell -1)\cup\big((J_0+n-1)\setminus \{n+\ell -1\}\big)\subseteq [2n-1].
$$
We proceed as above to obtain $r_S^{\ell}=c^2_{J_0}$. By associativity it is then sufficient to show that $r_S^{2\ell -1}=0$. Using the notation
of Lemma~\ref{lemma:8sdf7}, we can readily see that $|K_S^{2\ell -1}|\geqslant |J_0|$. Hence by Lemma~\ref{lemma:8sdf7} we only need to show
that $c_{K_S^{2\ell -1}}=0$. If $|K_S^{2\ell -1}|>|J_0|$, then we conclude by using the induction hypothesis. If $|K_S^{2\ell -1}|=|J_0|$, then
we can apply case $(i)$ since $\min(K_S^{2\ell -1})=1$.
\end{enumerate}
If $\ell >(n+1)/2$, we proceed symmetrically by setting $\ell'=\max(J)$ and considering the cases $\ell'=n$ and $(n+1)/2\leqslant\ell'<n$
separately.
\end{proof}

\begin{proposition}\label{linearassoc}
A polynomial function $p\colon R^n\to R$ of the form (\ref{eq:sdfa987}) with $c_{[n]}=0$ is associative if and only if it is one of the
functions of types $(i)$--$(v)$.
\end{proposition}

\begin{proof}
It is straightforward to see that the functions of types $(i)$--$(v)$ are associative polynomial functions.

Now, by Proposition~\ref{degone} the polynomial function $p$ has the form
\[
p(\bfx)=c_0+\sum_{i=1}^nc_i x_i\, ,\qquad c_0,\ldots,c_n\in R.
\]
Comparing the coefficients of $x_1$ in $p_1$ and $p_2$, we obtain the equation $c_1^2=c_1$. Similarly, we show that $c_n^2=c_n$. If $c_1=0$ or
$c_n=0$, we conclude by Proposition~\ref{prop:degrees}. Thus we can assume that $c_1=c_n=1$. Comparing the coefficients of $x_i$ in $p_i$ and
$p_{i-1}$ for $2\leqslant i\leqslant n$, we obtain the equations $c_1c_i=c_2c_{i-1}$, or equivalently, $c_{i}=c_2^{i-1}$ and $c_2^{n-1}=1$.
Finally, since the constant term in $p_i$ is $c_0+c_ic_0$, we must have $c_0=0$ unless $c_1=\cdots =c_n$.
\end{proof}

\begin{lemma}\label{lemma:sym}
Let $p\colon R^n\to R$ be an associative polynomial function of the form (\ref{eq:sdfa987}). If $c_{[n]}\neq 0$, then $p$ is a symmetric
function.
\end{lemma}

\begin{proof}
Let us first prove that $c_J=c_{J'}$ for every $J,J'\in [n]$ such that $|J|=|J'|=n-1$. Setting $S=[2n-1]\setminus\{n\}$, we see by
Lemma~\ref{lemma:8sdf7} that $r_S^i=c_{[n]}c_{[n]\setminus\{n-i+1\}}$ for $i\in[n]$ and we conclude by associativity.

We now proceed by induction. Suppose that $c_J=c_{J'}$ for every $J,J'\in [n]$ such that $|J|=|J'|\geqslant k$ for some $2\leqslant k\leqslant
n-1$ and set $c_{|J|}=c_J$ for every $J\subseteq [n]$ such that $|J|\geqslant k$. Fix $J_0$ such that $|J_0|=k-1$ and set $S=J_0\cup C_1$ and
$m=\min([n]\setminus J_0)\leqslant n-1$. By Lemma~\ref{lemma:8sdf7} and associativity we have
$c_{[n]}c_{J_0}=r_S^1=r_S^{m+1}=c_{n-1}c_{|J_0|+1}$.
\end{proof}

The interest of Lemma~\ref{lemma:sym} is shown by the following obvious result.

\begin{lemma}\label{symassoc}
A symmetric function $f\colon R^n\to R$ is associative if and only if the associated functions $f_1,\ldots,f_n$ satisfy the condition $f_1=f_2$.
\end{lemma}

% \begin{proof}
% ;;;This proof should not appear in the final version. We should find a reference;;;
% The condition is obviously necessary. Let us show that it is sufficient and consider $n\geqslant 3$.
%  For $i\in\{3,\ldots,n\}$, we use the symmetry, the hypothesis $p_1=p_2$ and then the symmetry again to obtain
%  \begin{multline}p_i(x_1,\ldots,x_{2n-1})=p(x_1,\ldots,x_{i-1},p(x_i,\ldots,x_{i+n-1}),x_{i+n},\ldots,x_{2n-1})\\
% =p(x_{i-1},p(x_i,\ldots,x_{i+n-1}),x_{i+n},\ldots,x_{2n-1},x_1,\ldots,x_{i-2})\\
% =p(p(x_{i-1},x_i,\ldots,x_{i+n-2}),x_{i+n-1},\ldots,x_{2n-1},x_1,\ldots,x_{i-2})\\
% p(x_1,\ldots,x_{i-2},p(x_{i-1},\ldots,x_{i+n-2}),x_{i+n-1},\ldots,x_{2n-1})\\
% =p_{i-1}(x_1,\ldots,x_{2n-1})
% \end{multline}
% and the result follows.
% \end{proof}

Recall that the $n$-variable \emph{elementary symmetric polynomial functions} of degree $k\leqslant n$ are defined by
\[
P_k(\bfx)=\sum_{K\subseteq [n],\, |K|=k}~\prod_{i\in K}x_i.
\]

\begin{proposition}
A polynomial function $p\colon R^n\to R$ such that $\mathrm{deg}(p)>1$ is associative if and only if it is of the form
\begin{equation}\label{condassoc}
p(\bfx)=\sum_{k=0}^nc_k\, P_k(\bfx),
\end{equation}
where the coefficients $c_k\in R$ satisfy the conditions
\begin{equation}\label{condpol}
c_{j+1}c_k+c_j\delta_{k,0}=c_jc_{k+1}\, ,\qquad j\in [n-1],~ k\in [n]-1.
\end{equation}
\end{proposition}

\begin{proof}
By Proposition~\ref{linearassoc} and Lemma~\ref{lemma:sym}, any associative polynomial function $p\colon R^n\to R$ such that $\mathrm{deg}(p)>1$
is of the form (\ref{condassoc}). By Lemma~\ref{symassoc}, such a polynomial function is associative if and only if $p_1=p_2$, that is, with the
notation of Lemma~\ref{lemma:8sdf7}, $r^1_S=r^2_S$ for every $S\subseteq [2n-1]$.

Set $j=|S\cap C_1|$, $k=|S\cap B_1|$, $j'=|S\cap A_2|+|S\cap C_2|$, and $k'=|S\cap B_2|$. We have either $j'=j-1$ and $k'=k+1$, or $j'=j+1$ and
$k'=k-1$, or $j'=j$ and $k'=k$. Therefore we get the equations
$$
\begin{array}{ll}
c_{j+1}c_k+c_j\delta_{k,0}=c_jc_{k+1}\, , & j\in[n-1],~ k\in[n]-1,\\
c_{j+2}c_{k-1}+c_{j+1}\delta_{k-1,0}=c_{j+1}c_{k}\, , & j\in[n-1]-1,~ k\in[n].
\end{array}
$$
We conclude by observing that both sets of conditions are equivalent.
\end{proof}

Let us now consider the special case where $R$ is a field.

\begin{proposition}\label{prop:field}
Assume that $R$ is a field. The associative polynomial functions from $R^n$ to $R$ of degree $>1$ are of the form
\begin{equation}\label{pab}
p_{a,b}(\bfx)=-b+a\prod_{i=1}^n(x_i+b),
\end{equation}
where $a\in R\setminus\{0\}$ and $b\in R$.
\end{proposition}

\begin{remark}
The functions $p_{a,b}$ defined in (\ref{pab}) can be written in several equivalent forms. It is easy to see that they are associative since so
are $p_{a,0}$ and $p_{a,b}=\varphi\circ p_{a,0}\circ(\varphi^{-1},\ldots,\varphi^{-1})$ where $\varphi(x)=x-b$.
\end{remark}

\begin{proof}
Since the coefficient $c_n$ in (\ref{condassoc}) is nonzero by Proposition~\ref{linearassoc}, we can set $a=c_n$ and $b=c_{n-1}/a$. Using equation (\ref{condpol}) for $j=n-1$ and $k\geqslant 1$, we obtain $c_k=b\,
c_{k+1}$, that is, $c_k=ab^{n-k}$. Using again equation (\ref{condpol}) for $j=n-1$ and $k=0$, we obtain $c_0=-b+ab^n$. We conclude by observing
that the function $p_{a,b}$ is associative (see remark above).
\end{proof}

We see from the proof of Proposition~\ref{prop:field} that the system of equations (\ref{condpol}) has a unique solution in $\mathrm{Frac}(R)$.
Therefore we can characterize the associative polynomial functions of degree $>1$ as the restrictions to $R$ of nonzero multiples of the
product function, up to an affine transformation in $\mathrm{Frac}(R)$.

\begin{proposition}\label{prop:prod}
Any associative polynomial function $p\colon R^n\to R$ such that $\mathrm{deg}(p)>1$ is of type $(vi)$.
\end{proposition}

%---------------------------------------------------------------------------------------------- Section 3
\section{Further properties}

We now investigate a few properties of the semigroup structures that we have determined.

\subsection{$n$-ary groups}

After classifying the ternary semigroups defined by polynomial functions, G{\l}azek and Gleichgewicht \cite{GG85} determined the corresponding
ternary groups. Using the Main Theorem, we can also derive a description of the $n$-ary groups defined by polynomial functions. Recall that an
\emph{$n$-ary quasigroup} is given by a nonempty set $G$ and an $n$-ary operation $f\colon G^n\to G$ such that for every $a_1,\ldots,a_n,b\in G$
and every $i\in [n]$ the equation
\begin{equation}\label{coucou1}
f(a_1,\ldots,a_{i-1},z,a_{i+1},\ldots,a_n)=b,
\end{equation}
has a unique solution in $G$. An \emph{$n$-ary group} is then an $n$-ary quasigroup $(G,f)$ that is also an $n$-ary semigroup. Recall also that
in an $n$-ary group, with any element $x$ is associated the element $\overline{x}$, called \emph{skew to $x$}, defined by the equation
$f(x,\ldots,x,\overline{x})=x$.

\begin{proposition}\label{narygroups}
The $n$-ary groups $(R,p)$ defined by polynomial functions $p\colon R^n\to R$ of degree $\leqslant 1$ are of type $(iv)$ with
$\overline{x}=(2-n)x-c$ and type $(v)$ with $\overline{x}=x$.
\end{proposition}

\begin{proof}
We immediately see that the polynomials of types $(i)$--$(iii)$ do not define $n$-ary groups. It is well known that the polynomials of types
$(iv)$ and $(v)$ define $n$-ary groups.\footnote{Polynomial functions of type $(v)$ were already considered by D\"ornte \cite[p.~5]{Do1928} in
the special case of complex numbers.}
\end{proof}

In general, the $n$-ary semigroups $(R,p)$ defined by type $(vi)$ are not $n$-ary groups. In the special case where $R$ is a field, we have the
following immediate result.

\begin{proposition}
If $R$ is a field, the $n$-ary semigroup $(R\setminus\{-b\},p_{a,b})$, where $p_{a,b}$ is defined in (\ref{pab}), is an $n$-ary group. It is
isomorphic to $(R\setminus\{0\},p_{a,0})$.
\end{proposition}

%;;;it would be nice to have the converse statement : If $p_{a,b}$ defines an $n$-ary group on $R\setminus\{-b\}$, then $R$ is a field;;;

\subsection{Medial $n$-ary semigroup structures}

We observe that all the $n$-ary semigroup structures given in the Main Theorem are medial. This is a general fact for functions of degree
$\leqslant 1$ on a commutative ring. This is also immediate for the function $p_{a,b}$ defined in (\ref{pab}) because it is the restriction to
$R$ of an $n$-ary operation that is isomorphic to a nonzero multiple of the $n$-ary product operation on $\mathrm{Frac}(R)$. From this
observation it follows that, for the $n$-ary groups given in Proposition~\ref{narygroups}, the map $x\mapsto\overline{x}$ is an endomorphism.

\subsection{(Ir)reducibility of $n$-ary semigroup structures}

Recall that if $(G,\circ)$ is a semigroup, then there is an obvious way to define an $n$-ary semigroup by $f(x_1,\ldots,x_n)=x_1\circ\cdots\circ
x_n$. In this case, the $n$-ary semigroup $(G,f)$ is said to be \emph{derived} from $(G,\circ)$ or \emph{reducible} to $(G,\circ)$, otherwise it
is said to be \emph{irreducible}. It is clear that the $n$-ary semigroups defined in types $(i)$--$(iii)$ are derived from the corresponding
semigroups. However, the $n$-ary semigroups defined in type $(v)$ are not reducible. Indeed, otherwise we would have $y=y\circ 0\circ\cdots\circ
0$ for all $y\in R,$ and therefore
\[x\circ y=x\circ (y\circ 0\circ\cdots\circ 0)= x\circ (y\circ 0)\circ 0\circ\cdots\circ 0=x+\omega (y\circ 0),\]
for $x,y\in R$, which leads to $x\circ y=x+\omega y+c$, where $c=\omega^2(0\circ 0)$. We know from the Main Theorem that this function does not
define a semigroup. We can prove similarly that the $n$-ary semigroups defined in type $(iv)$ are reducible if and only if $c=(n-1)c_0$ for
$c_0\in R$ and, when $R$ is a field, that the semigroup $(R\setminus\{0\},p_{a,0})$ is derived from a semigroup if and only if $a=a_0^{n-1}$ for
$a_0\in R$.

%---------------------------------------------------------------------------------------------- Acknowledgments
\section*{Acknowledgments}

This research is supported by the internal research project F1R-MTH-PUL-09MRDO of the University of Luxembourg.

%\bibliographystyle{abbrv}   % styles: plain, unsrt, alpha, abbrv, ieeetr, acm, siam, apalike, amsplain,...
%\bibliography{ReferencesMarichal}

\end{document}